\documentclass[reqno]{amsart}

\usepackage{amssymb}
\usepackage{amsthm}
\usepackage{amsmath}
\usepackage{amsfonts}
\usepackage{comment}
\usepackage{enumerate, enumitem}
\usepackage{tikz}

\newtheorem{thm}{Theorem}[section]
\newtheorem{lem}[thm]{Lemma}

\newtheorem{cor}[thm]{Corollary}
\newtheorem{theoremintro}{Theorem}
\newtheorem{corintro}{Corollary}

\newcommand{\bbold}{\mathbb}

\newcommand{\nn}{\mathbb{N}}
\newcommand{\qq}{\mathbb{Q}}
\newcommand{\zz}{\mathbb{Z}}
\def\R { {\bbold R} }
\def\Z { {\bbold Z} }
\def\T { {\bbold T} }

\def \d{\operatorname{d}}
\def \<{\langle}
\def \>{\rangle}

\def\Mo{\operatorname{Mo}}

\DeclareFontFamily{OMS}{smallo}{}
\DeclareFontShape{OMS}{smallo}{m}{n}{<->s*[.65]cmsy10}{}
\DeclareSymbolFont{smallo@m}{OMS}{smallo}{m}{n}
\DeclareMathSymbol{\smallo}{\mathord}{smallo@m}{79}

\DeclareFontFamily{U}{fsy}{}
\DeclareFontShape{U}{fsy}{m}{n}{<->s*[.9]psyr}{}
\DeclareSymbolFont{der@m}{U}{fsy}{m}{n}
\DeclareMathSymbol{\der}{\mathord}{der@m}{182}

\begin{document}

\title{Definability in differential-henselian monotone fields}

\begin{abstract}
This paper is a sequel to \cite{TH} and considers definability in differential-henselian monotone fields with c-map and angular component map. We prove an Equivalence Theorem among whose consequences are a relative quantifier reduction and an NIP result. 
\end{abstract}

\author[Hakobyan]{Tigran Hakobyan}
\address{Department of Mathematics\\
University of Illinois at Urbana-Cham\-paign\\
Urbana, IL 61801\\
U.S.A.}
\email{hakobya2@illinois.edu}

\maketitle

\begin{comment}

\begin{abstract} Scanlon~\cite{S} proves Ax-Kochen-Ersov type results for differential-henselian monotone valued differential
fields with many constants. We show how to get rid of
the condition {\em with many constants}. 
\end{abstract}
\end{comment}

\section{Introduction}

\medskip\noindent
Let $\mathbf{k}$ be a differential field (always of characteristic $0$ in this paper,
with a single distinguished derivation). Let also an ordered abelian group $\Gamma$ be given. This gives rise to the Hahn field 
$K=\mathbf{k}((t^\Gamma))$, to be considered in the usual way as a valued
field. Moreover, let an additive map $c:\Gamma\to \mathbf{k}$ be given. Then
the derivation $\der$ of $\mathbf{k}$ extends to a derivation $\der_c$ of $K$ with $\der_c(t^\gamma)=c(\gamma)t^\gamma$ for all $\gamma\in \Gamma$ by setting
$$ \der_c(\sum_{\gamma} a_{\gamma}t^{\gamma})\ :=\ \sum_\gamma \big(\der(a_\gamma) + c(\gamma) a_\gamma \big)t^\gamma.$$

\noindent
Let $K_c$ be the valued differential field $K$ with $\der_c$ as its distinguished derivation. 
Assume in addition that $\mathbf{k}$ is linearly surjective in the sense that for each nonzero linear differential
operator $A=a_0+a_1\der + \dots +a_n\der^n$ over $\mathbf{k}$ we have
$A(\mathbf{k})=\mathbf{k}$.
Then $K_c$ is differential-henselian and in \cite{TH} we considered these valued differential fields and among other things proved an Ax-Kochen-Ershov type theorem for them. Taking $c(\gamma)=0$ for all $\gamma$, this becomes the result of Scanlon in \cite{S}. The current paper is a sequel to \cite{TH} and focuses on definability, similar to Scanlon~\cite{S2}. 
%For this, we actually work in a more general setting described below.

More generally we consider, as in \cite[Section 6]{TH}, three-sorted structures
$$\mathcal{K}\ =\ (K, \mathbf{k}, \Gamma; \pi, v, c)$$
where $K$ and $\mathbf{k}$ are differential fields, $\Gamma$ is an ordered abelian group,
$v:K^\times \to \Gamma$ is a valuation which makes $K$ into a valued differential field with a valuation ring $\mathcal{O} := \mathcal{O}_v$ such that $K$ is monotone in the sense that $v(a') \geq v(a)$ for all $a\in K$,
$\pi:\mathcal{O} \to \mathbf{k}$ is a surjective differential ring morphism,
$c:\Gamma \to \mathbf{k}$ is an additive map satisfying
$\forall \gamma \exists x\ne 0  \big( v(x) = \gamma \ \&\  \pi(x^\dagger) = c(\gamma)\big)$.
To make the maps $\pi$ and $v$ total, we add a formal symbol $\infty$ to the sorts of the residue field and the value group as default values, i.e. $\pi(x) = \infty$ if and only if $x \in K \setminus \mathcal{O}$, and $v(x) = \infty$ if and only if $x = 0$.
We construe these $\mathcal{K}$ as $L_3$-structures for a natural $3$-sorted language $L_3$ (with unary function symbols for $\pi, v$ and $c$).
We have an obvious set $\textnormal{Mo}(c)$ of $L_3$-sentences whose models are exactly these $\mathcal{K}$. We refer to \cite[Section 6]{TH} for details, and just mention here one of the main results from \cite{TH}, namely that if $\mathcal{K}$ is
differential-henselian (as defined in the next section),
then $\textnormal{Th}(\mathcal{K})$ is completely axiomatized by 
$\textnormal{Mo}(c)$, axioms expressing differential-henselianity, and $\textnormal{Th}(\mathbf{k},\Gamma;c)$. 

To study definability we consider, as in Scanlon~\cite{S2},  expansions of such $\mathcal{K}$ by an  angular component map $\textnormal{ac} : K \to \mathbf{k}$ on $\mathcal{K}$. In Section 3  we introduce a suitable notion of  {\em angular component map on $\mathcal{K}$} (the definition is not completely obvious) and show its existence if $\mathcal{K}$ is $\aleph_1$-saturated. 

\medskip\noindent
To state our results precisely, we need to be more specific about the language. The language $L_3$ has three sorts: f (the main field sort), r  (the residue field sort), and
v (the value group sort); it contains an f-copy $L_{\text{f}} = \{0, 1, -, +, \cdot , \der\}$  of the language of differential rings, an r-copy
$L_{\text{r}}=\{0,1,-,+,\cdot,\bar{\der}\}$ of this language,  and a v-copy $L_{\text{v}} = \{\leq, 0, -, + \}$ of  the language of ordered abelian groups, with disjoint $L_{\text{f}}$, $L_{\text{r}}$, $L_{\text{v}}$.  It also has the unary function symbols $\pi, v, c$
of (mixed) sorts fr, fv, vr, respectively.  This completes the description of $L_3$. 
Let $L_{\textnormal{rv}}$ be the $2$-sorted sublanguage of $L_3$ consisting of $L_\textnormal{r}$ and $L_{\textnormal{v}}$,
 and the function symbol $c$. By $L_3(\textnormal{ac})$ we mean $L_3$ augmented by a new unary function symbol ac of sort fr. 
Let $T$ be the $L_3(\textnormal{ac})$-theory of $\d$-henselian monotone valued differential fields with angular component map.
Let the $L_3(\textnormal{ac})$-structures
$$\mathcal{K}_1\ =\ (K_1,\mathbf{k}_1,\Gamma_1;\ \pi_1, v_1, c_1, \textnormal{ac}_1), \qquad  \mathcal{K}_2\ =\ (K_2, \mathbf{k}_2, \Gamma_2;\ \pi_2, v_2, c_2, \textnormal{ac}_2)$$
be models of $T$. The main result of this paper is the Equivalence Theorem \ref{eqthm}
among whose consequences are the following (see Section 5):

\begin{theoremintro}\label{aa}
If $\mathcal{K}_1 \subseteq \mathcal{K}_2$ and $(\mathbf{k}_1, \Gamma_1; c_1) \preccurlyeq_{ L_{\textnormal{rv}}} (\mathbf{k}_2, \Gamma_2; c_2)$, 
then $\mathcal{K}_1 \preccurlyeq \mathcal{K}_2$.
\end{theoremintro}

\noindent
Here ``$\subseteq$'' means ``substructure of''. 
We also derive a relative quantifier reduction result. This uses a technical notion of {\it special formula} whose definition can be found in Section 5.
\begin{theoremintro}
Every $L_3(\textnormal{ac})$-formula is $T$-equivalent to a special $L_3(\textnormal{ac})$-formula.
\end{theoremintro}

\noindent
We use this to prove the following for models $\mathcal{K} = (K, \mathbf{k}, \Gamma; \pi, v, c, \textnormal{ac})$ of $T$:

\begin{corintro}\label{bb} If a set $X\subseteq \mathbf{k}^m\times \Gamma^n$ is definable in $\mathcal{K}$, then $X$ is already definable in the $L_{\textnormal{rv}}$-structure $(\mathbf{k}, \Gamma; c)$.
\end{corintro}

\begin{corintro}\label{cc}
$\mathcal{K}$ has $\textnormal{NIP}$ if and only if the 2-sorted structure $(\mathbf{k}, \Gamma; c)$ has $\textnormal{NIP}$.
\end{corintro}

\noindent
We also show how to eliminate the angular component maps from Theorem~\ref{aa} and Corollaries~\ref{bb} and ~\ref{cc}.
%\marginpar{need to say how to define $\pi$ and $v$ on all of $K$: add $\infty$ to r and v sorts as default values} 

\section{Preliminaries}\label{sec:pre} 

\medskip\noindent
A differential field is throughout a field of characteristic zero 
with a derivation on it. For an element $f$ of a differential field we
usually denote its derivative with respect to the given derivation by $f'$, and
also set $f^\dagger:= f'/f$ when $f\ne 0$.  

Adopting terminology from \cite{ADAMTT},
a {\em valued differential field\/} is
a differential field $K$ together with a
(Krull) valuation $v: K^{\times} \to \Gamma$ whose
residue field $\mathbf{k}$ 
$:=\mathcal{O}/\smallo$ has characteristic 
zero; here $\Gamma=v(K^\times)$ is 
the value group, and we also let $\mathcal{O}$ denote the valuation ring 
of $v$ with maximal ideal $\smallo$, and let
$$C\ =\ C_K\ :=\ \{f\in K:\ f'=0\}$$ 
denote the constant field of the differential field $K$, and 
for $a,b\in K$ we set $$a\asymp b:\Leftrightarrow v(a)=v(b),\quad a\preceq b:\Leftrightarrow v(a)\ge v(b),\quad a\prec b: \Leftrightarrow v(a) > v(b).$$ 
(If we wish to specify the ambient $K$ or the valuation $v$, we also write $\mathcal{O}_K$ or $\mathcal{O}_v$ to indicate the  valuation ring.)
Let $K$ be a valued differential field as above, and let $\der$ 
be its derivation. We say that $\der$ is {\em small\/} if 
$\der(\smallo)\subseteq \smallo$. In that case, $\der$ is continuous
with respect to the valuation topology on $K$; see \cite[4.4.6]{ADAMTT}.
We say that
$K$ has {\em many constants\/} if $v(C^\times)=\Gamma$.   
We say that
$K$ has {\em few constants\/} if $v(C^\times)=\{0\}$.
Following \cite{C} we call $K$ {\em monotone\/} if $f'\preceq f$ for all $f\in K$.
If $K$ is monotone, then $\der$ is obviously
small. If $K$ has many constants and small derivation, then $K$ is easily seen to be monotone. If 
$K$ is monotone, then so is any valued differential field extension 
with small derivation and the same value group as $K$; see
\cite[6.3.6]{ADAMTT}.

{\em From now on we assume that the derivation
of $K$ is small}. This has the effect (see \cite{C} or
\cite[Lemma 4.4.2]{ADAMTT}) that also 
$\der(\mathcal{O})\subseteq \mathcal{O}$, and so $\der$ induces a derivation
on the residue field; we view $\mathbf{k}$ below as
equipped with this induced derivation, and refer to it as the 
{\em differential residue field of $K$}.

\medskip\noindent
We say that $K$ is {\em differential-henselian\/} (for short: 
{\em $\operatorname{d}$-henselian}) if every differential polynomial 
$P\in \mathcal{O}\{Y\}=\mathcal{O}[Y, Y', Y'',\dots]$ whose reduction
$\overline{P}\in \mathbf{k}\{Y\}$ has total degree $1$ has a zero in 
$\mathcal{O}$. (Note that for ordinary 
polynomials $P\in \mathcal{O}[Y]$ this requirement
defines the usual notion of a henselian valued field, that is, 
a valued field whose valuation ring is henselian as a local ring.) 

If $K$ is $\operatorname{d}$-henselian, then its differential residue field
is clearly {\em linearly surjective}: any linear differential equation
$y^{(n)}+ a_{n-1}y^{(n-1)} + \cdots + a_0y = b$ with coefficients $a_i,b\in  \mathbf{k}$ has a solution in $\mathbf{k}$. This is a key
constraint on our notion of $\operatorname{d}$-henselianity. 
If $K$ is $\operatorname{d}$-henselian, then $\mathbf{k}$ has a {\em lift to $K$}, 
meaning, a differential subfield of $K$
contained in $\mathcal{O}$ that maps isomorphically 
onto $\mathbf{k}$ under the canonical map from $\mathcal{O}$ onto $\mathbf{k}$; see
\cite[7.1.3]{ADAMTT}. Other items from \cite{ADAMTT} that are relevant in this paper
are the following differential analogues of 
Hensel's Lemma and of results due to
Ostrowski/Krull/Kaplansky on valued fields:
\begin{enumerate}
\item[(DV1)] If the derivation of $\mathbf{k}$ is nontrivial, then
$K$ has a spherically complete immediate valued differential field extension 
with small derivation; \cite[6.9.5]{ADAMTT}. 
\item[(DV2)] If $\mathbf{k}$ is linearly surjective and 
$K$ is spherically complete, then $K$ is $\operatorname{d}$-henselian;
\cite[7.0.2]{ADAMTT}.

%\item[(DV3)] Suppose $K$ is $1$-$\d$-henselian and monotone, and 
%let $b \in K$, $b′ \prec b$.
%Then $b \asymp c$ for some $c \in C$;
%\cite[7.1.10]{ADAMTT}.

\item[(DV3)] If $\mathbf{k}$ is linearly surjective and $K$ is monotone, 
then any two spherically complete immediate monotone
valued differential field extensions of $K$ are isomorphic over $K$; 
\cite[7.4.3]{ADAMTT}.
\end{enumerate}

\noindent
We also need a model-theoretic variant of (DV3): 

\begin{enumerate} 
\item[(DV4)] Suppose $\mathbf{k}$ is linearly surjective and $K$ is monotone.
% with $v(K^\times) \not= \{ 0 \}$. 
Let $K^{\bullet}$ be a spherically complete immediate
valued differential field extension of $K$ with small derivation. Then  $K^{\bullet}$ can be embedded
over $K$ into any $|v(K^\times)|^+$-saturated $\d$-henselian monotone valued differential
field extension of $K$; \cite[7.4.5]{ADAMTT}.
\end{enumerate}
The list (DV1)--(DV4) almost repeats a similar list in \cite{TH}, but the assumption here that the derivation of
$K^{\bullet}$ in (DV4) is small was inadvertently omitted
there. The assumption there in (DV4) that $v(K^\times)\ne \{0\}$ is unnecessary and is dropped here.

\medskip\noindent
Consider 3-sorted structures
$$\mathcal{K}\ =\ (K, \mathbf{k}, \Gamma; \pi, v, c)$$
where $K$ and $\mathbf{k}$ are differential fields, $\Gamma$ is an ordered abelian group,
$v:K^\times \to \Gamma$ is a valuation which makes $K$ into a monotone valued differential field,
$\pi:\mathcal{O} \to \mathbf{k}$ with $\mathcal{O} := \mathcal{O}_v$ is a surjective differential ring morphism,
$c:\Gamma \to \mathbf{k}$ is an additive map satisfying
$\forall \gamma \exists x\ne 0 \big( v(x) = \gamma \ \&\  \pi(x^\dagger) = c(\gamma)\big)$.
We construe these $\mathcal{K}$ as $L_3$-structures for our $3$-sorted language $L_3$. We have an obvious set $\textnormal{Mo}(c)$ of $L_3$-sentences whose models are exactly these $\mathcal{K}$.

\begin{lem} \label{PropertyOfCWithDagger} Let $\mathcal{K} \models \Mo(c)$ be as above and $b\in K^\times$. Then $\pi(b^\dagger) = \pi(a^\dagger) + c(v(b))$ for some $a\asymp 1$ in $K$.
\end{lem}

\begin{proof} Take $x\in K^\times$ with $x\asymp b$ and $\pi(x^\dagger) = c(v(b))$. Then $a := b / x$ works.
\end{proof}

\section{Angular components}

\noindent
Let $\mathcal{K} = (K, \mathbf{k}, \Gamma; \pi, v, c) \models \Mo(c)$.
% and, is given such that $\textnormal{ker}(c) = v(C^\times)$.
%Let $\textnormal{st}:\mathcal{O} \to \mathbf{k}$ be the standard part map: it is given by $v(x - \textnormal{st}(x)) > 0$ for all $x\in \mathcal{O}$. 
%In other words, $\textnormal{st}$ is the composition
%$$ \mathcal{O} \to \textnormal{res}(K) \to \mathbf{k}$$  
%of the lifting map $\ell:\textnormal{res}(K) \to \mathbf{k}$ and
%of the residue map with the lifting.  
%Thus it is a differential ring morphism that is the identity on $\mathbf{k}$.
%Moreover, if $a\in C \cap \mathcal{O}$, then $\textnormal{st}(a) \in C_\mathbf{k} = C \cap \mathbf{k}$.

\medskip\noindent
An {\bf angular component map\/} on $\mathcal{K}$ is a map 
$ \textnormal{ac} : K^\times \to \mathbf{k}^\times$ such that
\begin{enumerate}[font=\normalfont]
\item $ \textnormal{ac}$ is a multiplicative group morphism and $ \textnormal{ac}(a) = \pi(a)$ for all $a \asymp 1$,
\item $\textnormal{ac}(a^\dagger) = \textnormal{ac}(a)^\dagger + c(v(a))$ for all $a\in K^\times$ with $a' \asymp a$,
%\item $\textnormal{ac}(d)\in C_\mathbf{k} \subseteq \mathbf{k}$ for all $d\in C^\times$.
\item $\textnormal{ac}(d)^\dagger = -c(v(d))$ for all $d\in C^\times$.
\end{enumerate}

\noindent
We extend such a map to all of $K$ by $\textnormal{ac}(0) = 0\in \mathbf{k}$, so (3) then becomes the equality
in (2) for $a'=0$ instead of $a'\asymp a$. Note that if $\mathcal{K}$ has few constants, then (1) and (2) together  imply (3). If $\mathcal{K}$ has many constants, then this notion of angular component map is easily seen to agree with that in \cite[Section 8.1]{ADAMTT}.

\begin{comment}

\begin{lem}
If there exists an angular component map on $\mathcal{K} \models \Mo(c)$, 
then $v(C^\times)\subseteq \textnormal{ker}(c)$. If $\mathcal{K}$
is $\d$-henselian, then the reverse inclusion holds.  
\end{lem}
\begin{proof} Suppose $\textnormal{ac}$ is an angular component map on 
$\mathcal{K}$. Let $d \in C^\times$ be given; towards deriving $c(v(d)) = 0$ we distinguish two cases:

\medskip\noindent
(i) The derivation on $\mathbf{k}$ is trivial. Pick $a \in K^\times$ with $a \asymp d$ and $\pi (a ^\dagger) = c(v(d))$. Then $a = d u$ with $u \asymp 1$, so $u' \prec 1$, and thus $c(v(d)) = \pi(a^\dagger) = \pi(u^\dagger) = 0$.

\noindent
(ii) The derivation on $\mathbf{k}$ is not trivial. Choose $a\in K^\times$
such that $a' \asymp a\asymp 1$. Then by condition (2) above
we have $\textnormal{ac}(a^\dagger) = \textnormal{ac}(a)^\dagger$. Also by (2),
\begin{align*}\textnormal{ac}(a^\dagger)\ &=\  
\textnormal{ac}((ad)^\dagger)\ =\  \textnormal{ac}(ad)^\dagger + c(v(ad))\\
&=\ \textnormal{ac}(a)^\dagger +\textnormal{ac}(d)^\dagger + c(v(d))\ =\ \textnormal{ac}(a)^\dagger + c(v(d)),
\end{align*} since $\textnormal{ac}(d)\in C_{\mathbf{k}}$ by (3). Therefore, $c(v(d)) = 0$.

Assume $\mathcal{K}$ is $\d$-henselian, and $c(\gamma)=0$. Take $a\in K^\times$
such that $v(a)=\gamma$ and $\pi(a^\dagger)=0$. Then $a^\dagger\prec 1$, so
by $\d$-henselianity we have $y\prec 1$ in $K$ with $y'+ a^\dagger(1+y)=0$,
that is $a(1+y)\in C$ and $v(a(1+y))=\gamma$.
\end{proof}
\end{comment}

Examples are Hahn differential fields $\mathbf{k}((t^\Gamma))_c$ (the differential residue field being identified with $\mathbf{k}$ in the usual way), with 
%, having the property $c(\Gamma) \cap \mathbf{k}^\dagger = \{ 0 \}$. Note that the last property is equivalent to $\textnormal{ker}(c) = v(C^\times)$ by \cite[Proposition 3.1]{TH}.
the angular component map given by $\textnormal{ac}(a) = a_{\gamma_0}$ for non-zero $a = \sum a_\gamma t^\gamma$ and $\gamma_0 = v(a)$. 

More generally, let $s: \Gamma \to K^\times$ be a cross-section with 
%$s(v(C^\times)) \subseteq C$ and
$\pi(s(\gamma)^\dagger)=c(\gamma)$ for all $\gamma$. We claim that this yields an angular component
map $\textnormal{ac}$ on $\mathcal{K}$ by 
$$\displaystyle \textnormal{ac}(a)\ :=\ \pi\Big(\frac{a}{s(v(a))}\Big) \text{ for }a\in K^\times.$$
Condition (1) is obviously satisfied. As to condition (2), first note that $\textnormal{ac}(s(\gamma)) = 1$ for $\gamma \in \Gamma$. Next, 
let $a\in K^\times$ and $a' \asymp a$. Then $a = s(v(a)) u$ where $u\asymp 1$ and hence ac$(a) = \pi(u)$. Also,
$$\displaystyle \pi(a^\dagger)\ =\ \pi\Big(s(v(a))^\dagger + u^\dagger \Big)\ =\ \pi\Big(s(v(a))^\dagger\Big) + \pi(u^\dagger)\ =\ c(v(a))+\pi(u^\dagger).$$
Using $a^\dagger \asymp 1$ and $u \asymp 1$, this yields ac$(a^\dagger) = c(v(a)) + \pi(u) ^ \dagger = \textnormal{ac}(a)^\dagger + c(v(a))$. As to (3), let $d\in C^\times$, so $d=u\cdot s(v(d))$ with $u\asymp 1$, hence $\textnormal{ac}(d)= \textnormal{ac}(u)\cdot\textnormal{ac}(s(v(d)))$ with $\textnormal{ac}(s(v(d))=1$, and thus 
$$\textnormal{ac}(d)^\dagger\ =\ \textnormal{ac}(u)^\dagger\ =\ \pi(u)^\dagger\ =\ \pi(u^\dagger).$$ 
Morever, $0=d^\dagger=u^\dagger + s(v(d))^\dagger$, so $u^\dagger=-s(v(d))^\dagger$, and thus $$\pi(u^\dagger)\ =\ -\pi\big(s(v(d))^\dagger\big)\ =\ -c(v(d)).$$ Therefore, $\textnormal{ac}(d)^\dagger=-c(v(d))$, as claimed.   
This leads to the following:

\begin{cor}\label{buildac}
Suppose the model $\mathcal{K}$ of $\Mo(c)$ is 
%$\d$-henselian and 
$\aleph_1$-saturated. Then there exists an angular component map on $\mathcal{K}$. 
\end{cor} 
\begin{proof} This is close to the proof of \cite[Theorem 4.5]{TH} and so we shall be brief. Let $G:=\{a\in K^\times:\ \pi(a^\dagger)=c(va)\}$, a subgroup of $K^\times$, definable in $\mathcal{K}$, with $v(G)=\Gamma$. Then
$H:=\mathcal{O}^\times \cap G$ is a pure subgroup of $G$, and so we get a cross-section $s: \Gamma\to K^\times$ with
$s(\Gamma)\subseteq G$ as in the proof of \cite[Lemma 4.4]{TH}. (The assumption in that Lemma of $\d$-henselianity is superfluous and the reference in its proof to \cite[Corollary 3.3.37]{ADAMTT} should be to \cite[Corollary 3.3.38]{ADAMTT}.)  Then $\pi(s(\gamma)^\dagger)=c(\gamma)$ for all
$\gamma\in \Gamma$. This gives an angular component map
$\textnormal{ac}$ on $\mathcal{K}$ by $\textnormal{ac}(a):=\pi\Big(a/s(v(a))\Big)$ for $a\in K^\times$.
\end{proof}

%Third, if $a\in C$, then $s(v(a))^\dagger = c(v(a)) = 0$ and thus 
%$s(v(a)) \in C$ as well. Therefore, $\displaystyle \textnormal{ac}(a) = \textnormal{st}\Big(\frac{a}{s(v(a))}\Big) \in C_\mathbf{k} = C \cap \mathbf{k}$, as $\displaystyle \frac{a}{s(v(a))} \in C \cap \mathcal{O}^\times$ and $\mathbf{k}$ is a lift of the differential residue field.

\section{Equivalence over substructures}
\medskip\noindent
In this section we consider 3-sorted structures 
$$\mathcal{K}\ =\ (K,\mathbf{k},\Gamma; \pi, v, c, \textnormal{ac})$$
where $(K,\mathbf{k},\Gamma; \pi, v, c)\models  \Mo(c)$ and
$\textnormal{ac} : K \to \mathbf{k}$ is an angular component map on  
$(K,\mathbf{k},\Gamma; \pi, v, c)$.
These 3-sorted structures are naturally $L_3(\textnormal{ac})$-structures
where the language $L_3(\textnormal{ac})$ is $L_3$ augmented by a unary
function symbol $\textnormal{ac}$ of sort fr. Let $\Mo(c, \textnormal{ac})$
be the set of $L_3(\textnormal{ac})$-sentences consisting of $\Mo(c)$
and a sentence expressing that $\textnormal{ac}$ is an angular component map
as defined in the previous section. Then these 3-sorted
structures are exactly the models of $\Mo(c, \textnormal{ac})$.
Given $\mathcal{K}$ as above we regard any subfield $E$ of $K$ as a 
valued subfield of $K$, so the valuation ring of such $E$ is 
$\mathcal{O}_E=E\cap \mathcal{O}_v$. 
We say that a differential subfield $E$ of $K$
{\em satisfies the $c$-condition\/} if for all $\gamma \in v(E^\times)$ there is an $x\in E^\times$ such that $v(x) = \gamma$ and $\pi(x^\dagger) = c(\gamma)$.
For example, $K$ satisfies the $c$-condition.

\medskip\noindent
Define a \textbf{good substructure} of $\mathcal{K}$ to be a triple $\mathbf{E} = (E, \mathbf{k}_{\mathbf{E}}, \Gamma_{\mathbf{E}})$ such that
\begin{enumerate}
\item $E$ is a differential subfield of $K$;
\item $\mathbf{k}_\mathbf{E}$ is a differential subfield of $\mathbf{k}$ with ac$(E)\subseteq \mathbf{k}_\mathbf{E}$ (and thus $\pi(\mathcal{O}_E)\subseteq \mathbf{k}_\mathbf{E}$);
\item $\Gamma_\mathbf{E}$ is an ordered abelian subgroup of $\Gamma$, with $v(E^\times)\subseteq \Gamma_\mathbf{E}$ and $c(\Gamma_\mathbf{E})\subseteq \mathbf{k}_\mathbf{E}$.
\end{enumerate}

\noindent
Note that we do {\em not} demand here that 
$\pi(\mathcal{O}_E)= \mathbf{k}_\mathbf{E}$ or $v(E^\times)= \Gamma_\mathbf{E}$. 
%Note that then $\textnormal{st}(\mathcal{O}_E) =E\cap \mathbf{k}_\mathbf{E}$ by (1) and (2).
For good substructures $\mathbf{E} = (E, \mathbf{k}_\mathbf{E}, \Gamma_\mathbf{E})$ and $\mathbf{F} = (F, \mathbf{k}_\mathbf{F}, \Gamma_\mathbf{F})$ of $\mathcal{K}$ we define $\mathbf{E} \subseteq \mathbf{F}$ to mean that $E\subseteq F$, $\mathbf{k}_\mathbf{E} \subseteq \mathbf{k}_\mathbf{F}$ and $\Gamma_\mathbf{E} \subseteq \Gamma_\mathbf{F}$. Now let $$\mathcal{K}_1\ =\ (K_1,\mathbf{k}_1,\Gamma_1; \pi_1, v_1, c_1, \textnormal{ac}_1), \quad  \mathcal{K}_2\ =\ (K_2, \mathbf{k}_2, \Gamma_2; \pi_2, v_2, c_2, \textnormal{ac}_2)$$
be  models of $\Mo(c, \textnormal{ac})$, set
$\mathcal{O}_1:=\mathcal{O}_{v_1} $ and $\mathcal{O}_2:=\mathcal{O}_{v_2} $,
and let $$\mathbf{E}_1\ =\ (E_1, \mathbf{k}_{\mathbf{E}_1}, \Gamma_{\mathbf{E}_1}), \qquad \mathbf{E}_2\ =\ (E_2, \mathbf{k}_{\mathbf{E}_2}, \Gamma_{\mathbf{E}_2})$$
be good substructures of $\mathcal{K}_1, \mathcal{K}_2$, respectively.
A \textbf{good map} $\mathbf{f}:\mathbf{E}_1 \to \mathbf{E}_2 $ is a triple $\mathbf{f}=(f, f_r, f_v)$ consisting of a differential field isomorphism $f: E_1 \to E_2$, a differential field isomorphism $f_r: \mathbf{k}_{\mathbf{E}_1} \to \mathbf{k}_{\mathbf{E}_2}$, and an ordered group isomorphism $f_v: \Gamma_{\mathbf{E}_1} \to \Gamma_{\mathbf{E}_2}$, such that
\begin{enumerate}
%\item[(4)] $f(a) = f_r(a)$ for all $a \in \textnormal{st}_1(\mathcal{O}_{E_1})$, 
%and $f^{-1}(b)=f_r^{-1}(b)$
% for all $b \in E_2\cap \mathbf{k}_{E_2}$; 
%$ \Leftrightarrow f(\textnormal{st}_1(a)) = \textnormal{st}_2(f(a))$ 
%for all $a \in E_1$,

\item[(4)] $f_r (\textnormal{ac}_1(a)) = \textnormal{ac}_2(f(a))$ for all $a\in E_1$;

\item[(5)] $f_v(v_1(a)) = v_2( f(a) )$ for all $a\in E_1^\times$;

\item[(6)] $(f_r, f_v)$ is elementary as a partial map between $(\mathbf{k}_1, \Gamma_1; c_1)$ and $(\mathbf{k}_2, \Gamma_2; c_2)$, in particular, $f_r(c_1(\gamma)) = c_2(f_v(\gamma))$ for all $\gamma \in \Gamma_{\mathbf{E}_1}$.
%such that $c_1(\gamma) \in \textnormal{dom}(f_r)$ and 
%$\gamma \in \textnormal{dom}(f_v)$.
\end{enumerate}

%If $x \in \mathbf{k}_{E_1} \cap E_1$, then $v_2(f(x)) = f_v(v_1(x)) = 0$, and thus $f(x) = y + \epsilon$, with $y = st_2(f(x)) \in \mathbf{k}_{E_2} \cap E_2$ and $\epsilon \in \smallo_{E_2}$. We can find $z \in E_1$ and $\delta \in \smallo_{E_1}$ with $f(z) = y$ and $f(\delta) = \epsilon$. Then $f(x) = y + \epsilon = f(z + \delta)$, so $x = z + \delta$

%then $f_r(x) = f_r(\textnormal{ac}_1(x)) = \textnormal{ac}_2(f(x))$

%In particular, the first condition means that $f$ and $f_r$ coinside on the intersection of their domains.

\noindent
Let  $\mathbf{f}=(f, f_r, f_v):\mathbf{E}_1 \to \mathbf{E}_2 $ be a good map. Then
$$\mathbf{f}^{-1}\ :=\ (f^{-1}, f_r^{-1}, f_v^{-1})\ :\ \mathbf{E}_2 \to \mathbf{E}_1$$
is a good map as well, and
$f(\mathcal{O}_{E_1})=\mathcal{O}_{E_2}$ by (5), so
$f$ is an isomorphism of valued fields. Using also (4) we obtain
$f_r(\pi_1(a))  =  \pi_2(f(a))$ for all $a \in \mathcal{O}_{E_1}$. 
 
We say that a good map $\mathbf{g} = (g, g_r, g_v) : \mathbf{F}_1 \to \mathbf{F}_2$ \textbf{extends} $\mathbf{f}$ if $\mathbf{E}_1\subseteq \mathbf{F}_1$, $\mathbf{E}_2 \subseteq \mathbf{F}_2$, and $g, g_r, g_v$ extend $f, f_r, f_v$, respectively. Note that if a good map $\mathbf{E}_1 \to \mathbf{E}_2$ exists, then $(\mathbf{k}_1, \Gamma_1; c_1) \equiv (\mathbf{k}_2, \Gamma_2; c_2)$ by (6). Our goal is:

\begin{thm}\label{eqthm}
If $\mathcal{K}_1$ and $\mathcal{K}_2$ are $\d$-henselian, then 
any good map $\mathbf{E}_1 \to \mathbf{E}_2$ is a partial elementary map between $\mathcal{K}_1$ and $\mathcal{K}_2$.
\end{thm}

\noindent
Towards the proof we establish some lemmas; these do not assume $\d$-henselianity.

\begin{lem} \label{ExtensionLemma1}
Let $\mathbf{f}:\mathbf{E}_1 \to \mathbf{E}_2$ be a good map and suppose $F_1 \supseteq E_1$ and $F_2 \supseteq E_2$ are differential subfields of $K_1$ and $K_2$, respectively, such that $\pi_1(\mathcal{O}_{F_1}) \subseteq \mathbf{k}_{\mathbf{E}_1}$ and $v_1(F_1^\times) = v_1(E_1^\times)$.
Let $g:F_1 \to F_2$ be a valued differential field isomorphism such that $g$ extends $f$ and $f_r(\pi_1(u)) = \pi_2(g(u))$ for all 
$u\asymp 1$ in  $F_1$. 
Then $\textnormal{ac}_1(F_1) \subseteq \mathbf{k}_{\mathbf{E}_1}$ and $f_r(\textnormal{ac}_1(a)) = \textnormal{ac}_2(g(a))$ for all $a\in F_1$, and thus also
$\textnormal{ac}_2(F_2) \subseteq \mathbf{k}_{\mathbf{E}_2}$.
\end{lem}

\begin{proof}
Let $a\in F_1$. Then $a = bu$ where $b \in E_1$ and $u\asymp 1$ in $F_1$, so $\textnormal{ac}_1(a) = \textnormal{ac}_1(b) \pi_1(u) \in \mathbf{k}_{\mathbf{E}_1}$.
Thus $f_r(\textnormal{ac}_1(a)) = f_r(\textnormal{ac}_1(b)) f_r(\pi_1(u)) = \textnormal{ac}_2(f(b)) \pi_2(g(u)) = \textnormal{ac}_2(g(a))$.
\end{proof}

\begin{lem} \label{ExtensionLemma3}
Suppose $\mathbf{f}:\mathbf{E}_1 \to \mathbf{E}_2$ is a good map,
$\pi_1(\mathcal{O}_{E_1}) = \mathbf{k}_{\mathbf{E}_1}$, and $F_1\supseteq E_1$, $F_2 \supseteq E_2$ are differential subfields of $K_1$ and $K_2$, respectively, such that $v_1(F_1^\times) = v_1(E_1^\times)$.
Let $g: F_1 \to F_2$ be a valued differential field isomorphism extending $f$, and $g_r: \pi_1(\mathcal{O}_{F_1}) \to \pi_2(\mathcal{O}_{F_2})$ the differential field isomorphism induced by $g$.
Then $\textnormal{ac}_1(F_1) = \pi_1(\mathcal{O}_{F_1})$ and $g_r(\textnormal{ac}_1(a)) = \textnormal{ac}_2(g(a))$ for all $a\in F_1$. Moreover, $\textnormal{ac}_2(F_2) = \pi_2(\mathcal{O}_{F_2})$ and $v_2(F_2^\times)=v_2(E_2^\times)$.
\end{lem}

\begin{proof}
Let $a\in F_1^\times$. Then $a = bu$ where $a\asymp b \in E_1^\times$ and $1 \asymp u \in \mathcal{O}_{F_1}$,  
so $\textnormal{ac}_1(a) = \textnormal{ac}_1(b) \pi_1(u) \in \pi_1(\mathcal{O}_{F_1})$.
It is clear that $g_r$ extends $f_r$.
Thus \begin{align*}
g_r(\textnormal{ac}_1(a))\ &=\ g_r(\textnormal{ac}_1(b))g_r(\pi_1(u))\ =\ f_r(\textnormal{ac}_1(b))\pi_2( g(u))\ =\ \textnormal{ac}_2(f(b)) \textnormal{ac}_2(g(u))\\ &=\ \textnormal{ac}_2(g(b)) \textnormal{ac}_2(g(u))\ =\ \textnormal{ac}_2(g(bu))\ =\ 
\textnormal{ac}_2(g(a)),
\end{align*}
as claimed. 
\end{proof}

\noindent
With the assumptions of the lemma above, $g_r$ extends $f_r$, 
$(F_1,\pi_1(\mathcal{O}_{F_1}), \Gamma_{\mathbf{E}_1})$ and  
$(F_2,\pi_2(\mathcal{O}_{F_2}), \Gamma_{\mathbf{E}_2})$ are good substructures 
of $\mathcal{K}_1$ and $\mathcal{K}_2$, respectively, and
$$\mathbf{g}\ =\ (g, g_r, f_v)\ :\ (F_1,\pi_1(\mathcal{O}_{F_1}), \Gamma_{\mathbf{E}_1}) \to (F_2,\pi_2(\mathcal{O}_{F_2}), \Gamma_{\mathbf{E}_2})$$ 
extends $\mathbf{f}$ and satisfies conditions (4) and (5) for good maps.

\begin{cor} \label{ExtensionLemma4}
Suppose $\mathbf{f}:\mathbf{E}_1 \to \mathbf{E}_2$ is a good map, $\pi_1(\mathcal{O}_{E_1}) = \mathbf{k}_{\mathbf{E}_1}$, and $F_1\supseteq E_1$ and $F_2 \supseteq E_2$ are differential subfields of $K_1$ and $K_2$, respectively, and are immediate extensions of $E_1$ and $E_2$, respectively. Let $g: F_1 \to F_2$ be a valued differential field isomorphism extending $f$. Then $\mathbf{g} = (g, f_r, f_v)$ is a good map that extends $\mathbf{f}$.
\end{cor}

\noindent
This follows by verifying the hypotheses of Lemma~\ref{ExtensionLemma3}. 

\begin{comment}

\begin{lem} \label{ExtensionLemma5}
Suppose $\textnormal{st}_1(\mathcal{O}_{E_1}) = \mathbf{k}_{E_1} \subseteq E_1$, $\mathbf{f}:\mathbf{E}_1 \to \mathbf{E}_2$ is a good map. Suppose $a\in \mathbf{k}_1 \setminus \mathbf{k}_{E_1}$ and $b\in \mathbf{k}_2 \setminus \mathbf{k}_{E_2}$ and let $g: E_1 \< a \> \to E_2 \< b \>$ be a valued differential field isomorphism extending $f$. Then $\mathbf{g} = (g, g_r, f_v)$ is a good map that extends $\mathbf{f}$.
\end{lem}
\end{comment}

\subsection*{Proof of Theorem~\ref{eqthm}} Assume $\mathcal{K}_1$ and
$\mathcal{K}_2$ are $\d$-henselian, and let 
$$\mathbf{f}\ =\ (f, f_r, f_v)\ :\ \mathbf{E}_1 \to \mathbf{E}_2$$ be a good map. We have to show that this is a 
partial elementary map between  $\mathcal{K}_1$ and
$\mathcal{K}_2$. The case 
$\Gamma_1 = \{0\}$ is a routine exercise, so assume $\Gamma_1 \neq \{0\}$.

By passing to suitable elementary extensions of $\mathcal{K}_1$ and $\mathcal{K}_2$ we arrange that both are $\kappa$-saturated, where $\kappa$ is an uncountable cardinal such that $|\mathbf{k}_{\mathbf{E}_1}|, |\Gamma_{\mathbf{E}_1}| < \kappa$.
We call a good substructure $\mathbf{E} = (E, \mathbf{k}_{\mathbf{E}}, \Gamma_\mathbf{E})$ of $\mathcal{K}_1$ \textbf{small} if $|\mathbf{k}_\mathbf{E}|, |\Gamma_\mathbf{E}| < \kappa$. We prove that the good maps whose domain $\mathbf{E}$ is small (such as the above $\mathbf{f}$) form a back-and-forth system from $\mathcal{K}_1$ to $\mathcal{K}_2$. This is enough to establish the theorem. 
We now present 8 procedures to extend a good map $\mathbf{f}$ as above:

\medskip\noindent
(1) \textit{Given $d \in \mathbf{k}_1$, arranging that $d \in \mathbf{k}_{\mathbf{E}_1}$.} This can be done by saturation without changing $f, f_v, E_1, \Gamma_{\mathbf{E}_1}$ by extending $f_r$ to a map with domain $\mathbf{k}_{\mathbf{E}_1}\<d\>$ which together with $f_v$ gives a partial elementary map between $(\mathbf{k}_1, \Gamma_1; c_1)$ and $(\mathbf{k}_2, \Gamma_2; c_2)$.

\medskip\noindent
(2) \textit{Given $\gamma \in \Gamma_1$, arranging that $\gamma\in \Gamma_{\mathbf{E}_1}$.} We can assume $c(\gamma) \in \mathbf{k}_{\mathbf{E}_1}$ by (1). Without changing $f, f_r$, $E_1$, $\mathbf{k}_{\mathbf{E}_1}$ we then use saturation as in (1) to extend $f_v$ to a map with domain $\Gamma_{\mathbf{E}_1} + \zz\gamma$ which together with
$f_r$ is a partial elementary map between $(\mathbf{k}_1, \Gamma_1; c_1)$ and $(\mathbf{k}_2, \Gamma_2; c_2)$.

\medskip\noindent
(3) \textit{Arranging $\mathbf{k}_{\mathbf{E}_1} = \pi_1( \mathcal{O}_{E_1})$.} Let $d\in \mathbf{k}_{\mathbf{E}_1}$ and $d\not \in \pi_1(\mathcal{O}_{E_1})$. Set $e = f_r(d)$. There are two possibilities:

(i) $d$ is $\d$-transcendental over $\pi_1(\mathcal{O}_{E_1})$. Then $e$ is $\d$-transcendental over $\pi_2(\mathcal{O}_{E_2})$. Take $a\in \mathcal{O}_1$ and $b\in \mathcal{O}_2$ with $\pi_1(a) = d$ and $\pi_2(b) = e$. Then by \cite[Lemma 6.3.1]{ADAMTT} we have $v_1(E_1^\times) = v_1(E_1 \< a \>^\times)$ and we get an isomorphism $g: E_1 \< a \> \to E_2 \< b \>$ of valued differential fields which extends $f$ and sends $a$ to $b$. Therefore by Lemma \ref{ExtensionLemma1}, $\mathbf{g} = (g, f_r, f_v)$ is a good map between $(E_1 \< a \>, \mathbf{k}_{\mathbf{E}_1}, \Gamma_{\mathbf{E}_1})$ and $(E_2 \< b \>, \mathbf{k}_{\mathbf{E}_2}, \Gamma_{\mathbf{E}_2})$ and it extends $\mathbf{f}$.

(ii) $d$ is $\d$-algebraic over $\pi_1(\mathcal{O}_{E_1})$.
Take a minimal annihilator $\bar{P} \in \pi_1 (\mathcal{O}_{E_1})\{Y\}$ of $d$ over $\pi_1 (\mathcal{O}_{E_1})$. Note that applying $f_r$ to the coefficients of $\bar{P}$ yields a minimal annihilator of $e$ over $\pi_2 (\mathcal{O}_{E_2})$.
Next, take $P\in \mathcal{O}_{E_1}\{Y\}$ such that applying
$\pi_1$ to the coefficients of $P$ yields $\bar{P}$ and such
that $P$ has the same complexity as $\bar{P}$. Now $\mathcal{K}_1$ is $\d$-henselian, so we obtain $a\in \mathcal{O}_1$ with $\pi_1(a) = d$ and $P(a)=0$. As in the proof of \cite[Lemma 7.1.4]{ADAMTT} one shows that $P$ is then a minimal annihilator of $a$ over $E_1$. Applying $f$ to the coefficients of $P$ yields $fP\in \mathcal{O}_{E_2}\{Y\}$, and as $\mathcal{K}_2$ is $\d$-henselian we obtain likewise an element $b\in \mathcal{O}_2$ with $\pi_2(b) = e$ and $fP(b)=0$; then $fP$ is again a minimal annihilator of $b$ over $E_2$. An argument in the beginning of the proof of \cite[Theorem 6.3.2]{ADAMTT} yields $v_1(E_1\< a \>^\times) = v_1(E_1^\times)$,  $v_2(E_2\< b \>^\times) = v_2(E_2^\times)$, and the uniqueness
part of that theorem then gives us a valued differential field isomorphism $g: E_1\<a\>\to E_2\<b\>$ that extends $f$ and sends $a$ to $b$; the same theorem also gives
$\pi_1(\mathcal{O}_{E_1\< a \>})=  \pi_1(\mathcal{O}_{E_1})\<d\>$ and $\pi_2(\mathcal{O}_{E_2\< b \>})=  \pi_2(\mathcal{O}_{E_2})\<e\>$.   
Using Lemma \ref{ExtensionLemma1} this yields a good map $\mathbf{g} = (g, f_r, f_v)$ extending $\mathbf{f}$ with small domain $(E_1\< a \>, \mathbf{k}_{E_1}, \Gamma_{E_1} )$.
%Likewise, $f_r(P) \in \pi_2({O}_{E_2}) \{ Y \}$ is a minimal annihilator of $e$ and, we can find $b\in \mathcal{O}_{E_2}$ with $\pi_2(b) = e$, $\pi_2(\mathcal{O}_{E_2\< b \>}) = \pi_2(\mathcal{O}_{E_2})\< e \>$ and $v(E_2\< b \>^\times) = v(E_2^\times)$.
%Thus, by \cite[Theorem 6.3.2]{ADAMTT} and Lemma \ref{ExtensionLemma2}, $\mathbf{g} = (g, f_r, f_v)$ is a good map between $(E_1 \< d \>, \mathbf{k}_{\mathbf{E}_1}, \Gamma_{\mathbf{E}_1})$ and $(E_2 \< e \>, \mathbf{k}_{\mathbf{E}_2}, \Gamma_{\mathbf{E}_2})$ with $g(d) = e$ and it extends $\mathbf{f}$.

\begin{comment}

Then $\frac{\der P}{\der Y^{(r)}}(d) \neq 0$.
Now take $Q\in (\mathcal{O}_{E_1})^\times \{Y\}$ such that $\textnormal{st}_1(Q) = P$.
Then, $Q(d) \prec 1$ and $\frac{\der Q}{\der Y^{(r)}}(d) \asymp 1$, so $Q$ is in dh-position at $d$ and the d-henselianity of $K_1$ gives an element $d + \epsilon \in K_1$ with $Q(d + \epsilon) = 0$, where $\epsilon \in \smallo_1$.
As all the coefficients of $Q$ are asymptotic to $1$, we get $Q(d) = 0$.
Moreover, any differential polynomial over $E_1$ with lower complexity than $Q$ cannot have $d$ as a root. Therefore, $Q$ is a minimal annihilator of $d$ over $E_1$.

Similarly, $f(Q) \in (\mathcal{O}_{E_2})^\times \{ Y\}$ is a minimal annihilator of $e$. And, thus, by Theorem 6.3.2 of \cite{ADAMTT} and Lemma \ref{ExtensionLemma1}, $\mathbf{g} = (g, f_r, f_v)$ is a good map between $(E_1 \< d \>, \mathbf{k}_{E_1}, \Gamma_{E_1})$ and $(E_2 \< e \>, \mathbf{k}_{E_2}, \Gamma_{E_2})$ with $g(d) = e$ and it extends $\mathbf{f}$.

\end{comment}

\medskip\noindent
By iterating the extension procedures in (i) and (ii) we complete step (3), that is, arrange $\mathbf{k}_{\mathbf{E}_1} = \pi_1(\mathcal{O}_{E_1})$. 

\begin{comment}

\medskip\noindent
(3) \textit{Arranging $\mathbf{k}_{E_1} = \textnormal{st}_1(\mathcal{O}_{E_1})$.} Let $d\in \mathbf{k}_{E_1}$ but $d\not \in \textnormal{st}_1(\mathcal{O}_{E_1})$. Set $e = f_r(d)$. The following two cases are possible:

(i) $d$ is $\d$-transcendental over $\textnormal{st}_1(\mathcal{O}_{E_1})$. Then $e$ is $\d$-transcendental over $\textnormal{st}_2(\mathcal{O}_{E_2})$. Then $v_1(E_1^\times) = v_1(E_1 \< d \>^\times)$ and by Lemma 6.3.1 of \cite{ADAMTT} we get an isomorphism $g: E_1 \< d \> \to E_2 \< e \>$ which extends $f$ and sends $d$ to $e$. Moreover, we have $\textnormal{st}_1(\mathcal{O}_{E_1 \< d \> }) = \textnormal{st}_1(\mathcal{O}_{E_1}) \< d \>$. Therefore by Lemma \ref{ExtensionLemma1}, $\mathbf{g} = (g, f_r, f_v)$ will be a good map between $(E_1 \< d \>, \mathbf{k}_{E_1}, \Gamma_{E_1})$ and $(E_2 \< e \>, \mathbf{k}_{E_2}, \Gamma_{E_2})$ and it extends $\mathbf{f}$.

(ii) $d$ is $\d$-algebraic over $\textnormal{st}_1(\mathcal{O}_{E_1})$. This case is similar to the previous case, but here we should take a minimal annihilator of $d$ over $\textnormal{st}_1(\mathcal{O}_{E_1})$ and then use Lemma 6.3.2 instead of Lemma 6.3.1 of \cite{ADAMTT}. As a result, we get a good map $\mathbf{g}$ which is an extension of $\mathbf{f}$.

By iterating these steps we complete the step (3), i.e. arrange $\mathbf{k}_{E_1} = \textnormal{st}_1(\mathcal{O}_{E_1})$.

\medskip
In the rest of the steps, we will assume that $\mathbf{k}_{E_1} = \textnormal{st}_1(\mathcal{O}_{E_1})$.
\end{comment}

\medskip\noindent
(4) \textit{Arranging that $\mathbf{k}_{\mathbf{E}_1}$ is linearly surjective and $\mathbf{k}_{\mathbf{E}_1} = \pi_1(\mathcal{O}_{E_1})$}. This is done by first iterating (1) and then applying (3).

\medskip\noindent
(5) \textit{Arranging that $E_1$ is $\d$-henselian as a valued differential field and $\mathbf{k}_{\mathbf{E}_1} = \pi_1(\mathcal{O}_{E_1})$.}
First apply (4) to arrange that
$\mathbf{k}_{\mathbf{E}_1}$ is linearly surjective and $\mathbf{k}_{\mathbf{E}_1} = \pi_1(\mathcal{O}_{E_1})$.
Next, use (DV1)--(DV4) to pass to spherically complete immediate extensions of $E_1$ and $E_2$ inside $K_1$ and $K_2$, and use Corollary \ref{ExtensionLemma4}.

\medskip\noindent
(6) \textit{Arranging that $E_1$ satisfies the $c$-condition.} 
Let $\gamma \in v_1(E_1^\times)$. Take $b \in E_1^\times$ with $v_1(b) = \gamma$. Then Lemma \ref{PropertyOfCWithDagger} gives $a \asymp 1$ in $K_1^\times$ with $\pi_1(b^\dagger) = \pi_1(a^\dagger) + c_1(\gamma)$. Using (1) we arrange $\pi_1(a) \in \mathbf{k}_{\mathbf{E}_1}$. Next, use (3) to arrange $\mathbf{k}_{\mathbf{E}_1} = \pi_1(\mathcal{O}_{E_1})$. 
Now choose $a^\star \in \mathcal{O}_{E_1}$ with $\pi_1(a^\star) = \pi_1(a)$. Then $\pi_1(a^\dagger) = \pi_1(a)^\dagger = \pi_1(a^\star)^\dagger$, so $\pi_1(b^\dagger) = \pi_1({a^\star})^\dagger + c_1(\gamma)$. Thus $x:=b/a^\star$ satisfies $v_1(x)=\gamma$ and $\pi_1(x^\dagger)=c_1(\gamma)$. This takes care of a single
$\gamma$, and doing the above iteratively we can deal with
all $\gamma\in v_1(E_1^\times)$, preserving $|\mathbf{k}_{\mathbf{E}_1}| < \kappa$; this process does not change
$\Gamma_{\mathbf{E}_1}$.

\medskip\noindent
In steps (1)--(6) the value group $v_1(E_1^\times)$ does not change, so if the first field $E_1$ of the domain $\mathbf{E}_1$ of our good map $\mathbf{f}$ satisfies the $c$-condition, then so does the first field of the domain of the extension of $\mathbf{f}$ constructed in each of (1)--(6).

%\medskip\noindent
%(7) \textit{Arranging that $E_1$ is $\d$-henselian as a valued differential field and satisfies the $c$-condition and $\mathbf{k}_{\mathbf{E}_1} = \pi_1(\mathcal{O}_{E_1})$.}
%This is done by first applying (6) and then (5).

In steps (7) and (8) below we assume that the domain $\mathbf{E}_1$ of our good map $\mathbf{f}$ has the following properties: $E_1$ is $\d$-henselian, $E_1$ satisfies the $c$-condition, and $\mathbf{k}_{\mathbf{E}_1} = \pi_1(\mathcal{O}_{E_1})$. Note that then
the codomain $\mathbf{E}_2$ has the corresponding properties. In view of these properties of $\mathbf{E}_i$ we have
$$ \big(E_i, \mathbf{k}_{\mathbf{E}_i}, v_i(E_i^\times);\
\pi_i|_{\mathcal{O}_{E_i}}, v_i|_{E_i^\times}, c_i|_{v_i(E_i^\times)}\big)\models \Mo(c) \qquad(i=1,2).$$
In (7) and (8) below we construct models of the theory $\Mo(\ell, c)$ defined in the beginning of
 \cite[Section 5]{TH}. Concerning $\Mo(\ell,c)$, we shall use \cite[Lemma 6.1]{TH}, which we slightly reformulate here for the convenience of the reader and because the statement in \cite{TH} has misprints:

\medskip\noindent 
 {\em Suppose $(K, \mathbf{k}, \Gamma; \pi, v, c)\models \Mo(c)$ is $\d$-henselian, and 
 $\iota: \mathbf{k}\to K$ is a lifting of the differential residue field $\mathbf{k}$ to $K$, that is, a differential field embedding with image in $\mathcal{O}$ such that $\pi(\iota(x))=x$ for all $x\in \mathbf{k}$. Then
 $\Big(\big(K,\iota(\mathbf{k})\big), \Gamma; v,\iota\circ c\Big)\models \Mo(\ell, c)$}.

\medskip\noindent
(7) \textit{Towards $v_1(E_1^\times) = \Gamma_{\mathbf{E}_1}$; the case of no torsion modulo $v_1(E_1^\times)$.}
Suppose $\gamma \in \Gamma_{\mathbf{E}_1}$ has no torsion modulo $v_1(E_1^\times)$, that is, $n\gamma \not \in v_1(E_1^\times)$ for all $n \geq 1$.
Take a lifting $\iota_1$ of the differential residue field $\mathbf{k_{\mathbf{E}_1}}$ to $E_1$. Then
$$ \iota_2\ :=\ f\circ \iota_1\circ f_r^{-1}\ :\ \mathbf{k_{\mathbf{E}_2}}\to E_2$$ is 
a lifting of the differential residue field $\mathbf{k_{\mathbf{E}_2}}$ to $E_2$. The proof of
\cite[Proposition 7.1.3]{ADAMTT} shows how to extend $\iota_i$ to a lifting, to be denoted also by $\iota_i$, of $\mathbf{k}_i$ to $K_i$ for $i = 1, 2$. Then by the above formulation of \cite[Lemma 6.1]{TH}, 
$$(*)\qquad \qquad\Big(\big(K_i, \iota_i(\mathbf{k}_i)\big), \Gamma_i; v_i, \iota_i \circ c_i\Big) \models \Mo(\ell, c) \qquad(i = 1, 2).\quad$$
It follows that we can take $a\in K_1^\times$ such that $v_1(a) = \gamma$ and $a^\dagger = (\iota_1 \circ c_1)(\gamma)$. We distinguish two cases:

\medskip\noindent
(i)$\ c_1(\gamma) = 0$. Then $a \in C_{K_1}$ and so $\textnormal{ac}_1(a)^\dagger = - c_1(\gamma) = 0$. Replace $a$ with $a/\iota_1(\textnormal{ac}_1(a))$.

\medskip\noindent
(ii)$\ c_1(\gamma) \neq 0$. Then $a' \asymp a$, so $\textnormal{ac}_1(a^\dagger) = \textnormal{ac}_1(a)^\dagger + c_1(\gamma)$. On the other hand, $\textnormal{ac}_1(a^\dagger) = \pi_1(a^\dagger) = c_1(\gamma)$.
So $\textnormal{ac}_1(a)^\dagger = 0$ and we replace $a$ with $a/\iota_1(\textnormal{ac}_1(a))$.

\medskip\noindent
In both cases the updated $a$ still satisfies $v_1(a) = \gamma$ and  $a^\dagger = (\iota_1 \circ c_1)(\gamma)$;
in addition, we have $\textnormal{ac}_1(a) = 1$.
Note that $a$ is transcendental over $E_1$ and $P(a) = 0$ where $$P(Y)\ :=\ Y' - (\iota_1 \circ c_1)(\gamma)Y\ \in\ \mathcal{O}_{E_1} \{ Y \}.$$
In the same way, we get $b \in K_2^\times$ with $v_2(b) = f_v(\gamma)$, $b^\dagger = (\iota_2 \circ c_2)(f_v(\gamma))$ and $\textnormal{ac}_2(b) = 1$. Then $b$ is transcendental over $E_2$ and $P^f(b) = 0$ where the differential polynomial $P^f(Y)\in \mathcal{O}_{E_2}\{Y\}$ is given by 
$$P^f(Y)\ :=\ Y' - (f\circ\iota_1 \circ c_1)(\gamma)Y\ 
=\ Y' - (\iota_2\circ f_r \circ c_1)(\gamma)Y 
=\ Y' - (\iota_2 \circ c_2)(f_v(\gamma))Y.$$
Then \cite[Lemma 3.1.30]{ADAMTT} gives 
$$v_1(E_1(a)^\times)\ =\ v_1(E_1^\times) + \Z\gamma\ \subseteq\ \Gamma_{\mathbf{E}_1}, \qquad \pi_1(\mathcal{O}_{E_1(a)})\ =\ \pi_1(\mathcal{O}_{E_1}).$$ It also yields a valued field isomorphism $g: E_1(a) \to E_2(b)$ extending $f$ with $g(a) = b$. Note that $g$ is in addition a differential field isomorphism. It is now routine to verify that $(g, f_r, f_v)$ is a good map with domain $(E_1(a), \mathbf{k}_{\mathbf{E}_1}, \Gamma_{\mathbf{E}_1})$. Using that $E_1$ satisfies the $c$-condition, it follows easily that $E_1(a)$ does as well. 

Next we pass to immediate extensions of $E_1(a)$ and $E_2(b)$ using (DV1)--(DV4) and appeal to Corollary~\ref{ExtensionLemma4} to obtain a good map 
$(h, f_r, f_v)$ extending $(g, f_r, f_v)$ whose domain
$(F_1, \mathbf{k}_{\mathbf{E}_1}, \Gamma_{\mathbf{E}_1})$
is such that $F_1$ is $\d$-henselian, $F_1$ satisfies the $c$-condition, and $\mathbf{k}_{\mathbf{E}_1} = \pi_1(\mathcal{O}_{F_1})$. What we have gained is that
$\gamma\in v_1(F_1^\times)$.

\medskip\noindent
(8) \textit{Towards $v_1(E_1^\times) = \Gamma_{\mathbf{E}_1}$; the case of prime torsion modulo $v_1(E_1^\times)$.}
Suppose $\gamma \in \Gamma_{\mathbf{E}_1} \setminus v_1(E_1^\times)$ and $l\gamma \in v_1(E_1^\times)$ where $l$ is a prime number.
Let $\iota_1$ and $\iota_2$ be as in (7).
Using $(*)$ as in (7) we obtain $a \in K_1^\times$ such that $v_1(a) = \gamma$, $a^\dagger = (\iota_1 \circ c_1)(\gamma)$, and $\textnormal{ac}_1(a) = 1$.
From \cite[Lemma 6.1]{TH} as formulated above we also get
$$(**) \qquad \Big(\big(E_i, \iota_i(\mathbf{k}_{\mathbf{E}_i})\big), v_i(E_i^\times);\ v_i|_{E_i^\times}, \iota_i\circ c_i|_{v_i(E_i^\times)}\Big)\models \Mo(\ell, c) \qquad(i=1,2).$$
This yields an element $b\in E_1^\times$ such that $v_1(b) = l\gamma$ and $b^\dagger = l(\iota_1 \circ c_1)(\gamma)$,
and as in (7) we can arrange in addition that $\textnormal{ac}_1(b) = 1$.
Our next aim is to find $d \in K_1^\times$ such that $d^l = b$ and $\textnormal{ac}_1(d) = 1$. To this end we 
consider $P(Y) := Y^l - b/a^l \in \mathcal{O}_1[Y]$. From $\textnormal{ac}_1(b/a^l) = 1$ and $b / a^l \asymp 1$ we get $v_1(1 - b/a^l) > 0$, that is, $v_1(P(1)) > 0$. Moreover, $P'(1) = l$. By henselianity we get $u\in K_1$ such that $P(u) = 0$ and $v_1(u - 1) > 0$. Setting $d: = a u\in K_1^\times$, we have 
$$d^l\ =\ b, \qquad \pi_1(d^\dagger)\ =\ c_1(\gamma), \qquad v_1(d)\ =\ \gamma, \qquad \textnormal{ac}_1(d)\ =\ \textnormal{ac}_1 (a)\textnormal{ac}_1(u)\ =\ 1.$$ 
Now $v_2(f(b))=lf_v(\gamma)$, $f(b)^\dagger=l(\iota_2\circ c_2)(f_r(\gamma))$, and $ac_2(f(b))=1$, so in the same way we constructed $d$, we find $e \in K_2^\times$ such that 
$$e^l\ =\ f(b), \qquad \pi_2(e^\dagger)\ =\ c_2(f_r(\gamma)),\qquad
v_2(e)\ =\ f_r(\gamma), \qquad  \textnormal{ac}_2(e)\ =\ 1.$$ Then \cite[Lemma 3.1.28]{ADAMTT} gives us an isomorphism $g: E_1(d) \to E_2(e)$ of valued fields extending $f$ and sending $d$ to $e$. 
This isomorphism is also a differential field isomorphism. Using that same lemma it is routine to check that 
$\mathbf{g}:=(g, f_r, f_v)$ is a good map with domain $(E_1(d), \mathbf{k}_{\mathbf{E}_1}, \Gamma_{\mathbf{E}_1})$. 
Using that $E_1$ satisfies the $c$-condition and $\pi(d^\dagger)=c_1(\gamma)$, it follows that $E_1(d)$ satisfies the $c$-condition. 

Unlike in (7) we do not need to extend further to immediate extensions of $E_1(d)$ and $E_2(e)$ to regain $\d$-henselianity: By \cite[Corollary 1.3]{TH},
$E_1(d)$ and $E_2(e)$ are $\d$-henselian. What we have gained is that $\gamma\in v_1(E_1(d)^\times)$.

%\medskip\noindent
%Iterating and alternating (7) and (8) we can now also achieve $v_1(E_1^\times) = \Gamma_{\mathbf{E}_1}$.

\bigskip\noindent
Now let any $a\in K_1$ be given; we need to extend $\mathbf{f}$ to a good map with small domain and with $a$ in its domain.
Using (1)--(8) we arrange that $E_1$ is $\d$-henselian, $E_1$ satisfies the $c$-condition, $\mathbf{k}_{\mathbf{E}_1} = \pi_1(\mathcal{O}_{E_1})$, and $v_1(E_1^\times) = \Gamma_{\mathbf{E}_1}$. Using that $E_1\<a\>$ has countable
transcendence degree over $E_1$ it follows from \cite[Lemma 3.1.10]{ADAMTT} that
$|\pi_1(\mathcal{O}_{E_1 \< a \>})| < \kappa$ and $ | v_1(E_1 \< a \> ^\times) | < \kappa$.
Using again (1)--(8) we extend $\mathbf{f}$ to a good map $\mathbf{f}_1 = (f_1, f_{1,r}, f_{1,v})$ with small domain $\mathbf{E}_1^1 = (E_1^1, \mathbf{k}_{\mathbf{E}_1^1}, \Gamma_{\mathbf{E}_1^1})$ such that $E_1^1$ is $\d$-henselian, satisfies the $c$-condition, and 
$$\pi_1(\mathcal{O}_{E_1 \< a \>})\ \subseteq\ \mathbf{k}_{\mathbf{E}_1^1}\ =\ \pi_1(\mathcal{O}_{E^1_1}), \qquad v_1(E_1 \< a \> ^\times)\ \subseteq\ \Gamma_{\mathbf{E}_1^1}\ =\ v_1((E_1^1)^\times).$$
Next we extend $\mathbf{f}_1$ in the same way to $\mathbf{f}_2 = (f_2, f_{2,r}, f_{2,v})$ with small domain $\mathbf{E}_1^2 = (E_1^2, \mathbf{k}_{\mathbf{E}_1^2}, \Gamma_{\mathbf{E}_1^2})$ such that $E_1^2$ is $\d$-henselian, $E_1^2$ satisfies the $c$-condition, and 
$$\pi_1(\mathcal{O}_{E_1^1 \< a \>})\ \subseteq\ \mathbf{k}_{\mathbf{E}_1^2}\ =\ \pi_1(\mathcal{O}_{ E_1^2}), \qquad v_1(E_1^1 \< a \> ^\times)\ \subseteq\ \Gamma_{\mathbf{E}_1^2}\ =\ v_1((E_1^2)^\times).$$
Continuing in this manner and taking the union of the resulting good maps and small domains, we get a good map $\mathbf{f}_\infty = (f_\infty, f_{\infty,r}, f_{\infty,v})$ with small domain $\mathbf{E}_1^\infty = (E_1^\infty, \mathbf{k}_{\mathbf{E}_1^\infty}, \Gamma_{\mathbf{E}_1^\infty})$ and codomain $\mathbf{E}_2^\infty = (E_2^\infty, \mathbf{k}_{\mathbf{E}_2^\infty}, \Gamma_{\mathbf{E}_2^\infty})$ such that $E_1^\infty$ is $\d$-henselian, $E_1^\infty$ satisfies the $c$-condition,  and
$$\pi_1(\mathcal{O}_{E_1^\infty \< a \>})\ =\ \mathbf{k}_{\mathbf{E}_1^\infty}\  =\ \pi_1(\mathcal{O}_{E_1^\infty}), \qquad v_1(E_1^\infty \< a \> ^\times)\ =\ \Gamma_{\mathbf{E}_1^\infty}\ =\ v_1((E_1^\infty)^\times).$$
Therefore, the differential valued field extension $E_1^\infty \< a \>$ of $E_1^\infty$ is immediate. By (DV1) and (DV4) we have a spherically complete immediate valued differential field extension $E_i^\bullet\subseteq K_i$ of $E_i^\infty \< a \>$ (and thus of $E_i$) for $i=1,2$.  
Then by (DV3) and Corollary~\ref{ExtensionLemma4} we can extend $\mathbf{f}_\infty$ to a good map $\mathbf{f}_\bullet$ with small domain $(E_1^\bullet, \mathbf{k}_{\mathbf{E}_1^\infty}, \Gamma_{\mathbf{E}_1^\infty})$ 
and codomain $(E_2^\bullet, \mathbf{k}_{\mathbf{E}_2^\infty}, \Gamma_{\mathbf{E}_2^\infty})$. It remains to note that $a\in E_1\<a\>\subseteq E_1^\bullet$.

This finishes the proof of the \textit{forth} part. The \textit{back} part is done likewise.

\section{Relative quantifier elimination}

\noindent
In this section we derive various consequences of Theorem \ref{eqthm}. 
Recall from the introduction the 3-sorted languages $L_3$ and $L_3(ac)$ and the 2-sorted language $L_{\textnormal{rv}}$. We also defined there $T$ to be the $L_3(\textnormal{ac})$-theory of d-henselian monotone valued differential fields with angular component map as
defined in Section 3. Let
$$\mathcal{K}_1\ =\ (K_1,\mathbf{k}_1,\Gamma_1;\ \pi_1, v_1, c_1, \textnormal{ac}_1), \qquad  \mathcal{K}_2\ =\ (K_2, \mathbf{k}_2, \Gamma_2;\ \pi_2, v_2, c_2, \textnormal{ac}_2)$$
be models of $T$ considered as $L_3(\textnormal{ac})$-structures.

\begin{cor}
$\mathcal{K}_1 \equiv \mathcal{K}_2$ if and only if $(\mathbf{k}_1, \Gamma_1; c_1) \equiv (\mathbf{k}_2, \Gamma_2; c_2)$ as $L_{\textnormal{rv}}$-structures.
\end{cor}
\begin{proof}
Suppose $(\mathbf{k}_1, \Gamma_1; c_1) \equiv (\mathbf{k}_2, \Gamma_2; c_2)$. Then we have good substructures $\mathbf{E}_1 = (\qq, \qq; \{ 0 \})$ of $\mathcal{K}_1$, $\mathbf{E}_2 = (\qq, \qq; \{ 0 \})$ of $\mathcal{K}_2$, and an obviously good map $\mathbf{E}_1 \to \mathbf{E}_2$, so Theorem \ref{eqthm} applies.
The other direction of the corollary is trivial.
\end{proof}

\begin{cor} \label{elementarysubstr}
Let $\mathcal{K}_1$ be a substructure of $\mathcal{K}_2$, with $(\mathbf{k}_1, \Gamma_1; c_1) \preccurlyeq (\mathbf{k}_2, \Gamma_2; c_2)$ as $L_{\textnormal{rv}}$-structures. Then $\mathcal{K}_1 \preccurlyeq \mathcal{K}_2$.
\end{cor}

\begin{proof} With $(K_1, \mathbf{k}_1, \Gamma_1)$ in the role of a good substructure of $\mathcal{K}_1$ as well as of $\mathcal{K}_2$, the identity on $(K_1, \mathbf{k}_1, \Gamma_1)$ is a good map. Hence $\mathcal{K}_1 \preccurlyeq \mathcal{K}_2$ by Theorem \ref{eqthm}.
\end{proof}

\noindent
To eliminate angular components in Corollary \ref{elementarysubstr}, consider $L_3$-structures
$$\mathcal{E}\ =\ (E,\mathbf{k}_E,\Gamma_E;\ \pi_E, v_E, c_E), \qquad  \mathcal{F}\ =\ (F,\mathbf{k}_F,\Gamma_F;\ \pi_F, v_F, c_F)$$
that are $\d$-henselian models of $\Mo(c)$.

\begin{cor}
Suppose $\mathcal{E}$ is a substructure of $\mathcal{F}$
and $(\mathbf{k}_{E}, \Gamma_{E}; c_E) \preccurlyeq (\mathbf{k}_F, \Gamma_F; c_F)$ as $L_{\textnormal{rv}}$-structures. Then $\mathcal{E} \preccurlyeq \mathcal{F}$.
\end{cor}

\begin{proof} By passing to suitable elementary extensions we arrange that $\mathcal{E}$ and $\mathcal{F}$ are 
$\aleph_1$-saturated.
Then the proof of Corollary~\ref{buildac} yields a cross-section
$s_E: \Gamma_E\to E^\times$ such that $\pi_E(s_E(\gamma)^\dagger)=c_E(\gamma)$ for all $\gamma\in \Gamma_E$, and also a cross-section
$s_F: \Gamma_F\to F^\times$ such that $\pi_F(s_F(\gamma)^\dagger)=c_F(\gamma)$ for all $\gamma\in \Gamma_F$. Now $\Gamma_E$ is an $\aleph_1$-saturated pure subgroup of $\Gamma_F$ and thus we have an internal direct sum decomposition $\Gamma_F=\Gamma_E \oplus \Delta$ by \cite[Corollary 3.3.38]{ADAMTT}. This gives a cross-section $s: \Gamma_F \to F^\times$ that agrees with $s_E$ on 
$\Gamma_E$ and with $s_F$ on $\Delta$. Moreover, $\pi_F(s(\gamma)^\dagger)=c_F(\gamma)$ for all $\gamma\in \Gamma_F$. This yields
an angular component map $\textnormal{ac}_F$ on $\mathcal{F}$ by $\textnormal{ac}(x) = \pi_F \big(x/s(v_F(x))\big)$. Its restriction to $\Gamma_E$ is angular component map
 on $\mathcal{E}$. Now use  Corollary \ref{elementarysubstr}.  
\end{proof}

\medskip\noindent
Let $x$ be an $l$-tuple of distinct f-variables, $y$ an $m$-tuple of distinct r-variables, and $z$ an $n$-tuple of distinct v-variables,
%Let $x \in K^l$, $y\in \mathbf{k}^m$, $z\in \Gamma^n$, $u\in \mathbf{k}^p$ and $\gamma\in \Gamma^q$.
Call an $L_3(\textnormal{ac})$-formula $\phi(x, y, z)$ {\em special\/} if $$\phi(x, y, z)\ =\ \psi\big(\textnormal{ac}(P_1(x)), \ldots, \textnormal{ac}(P_p(x)), v(Q_1(x)), \ldots, v(Q_q(x)), y, z \big),$$
for some $L_{\textnormal{rv}}$-formula 
$\psi\big(u_1, \ldots, u_p, w_1,\dots, w_q, y, z\big)$ where $u_1,\dots, u_p$ are extra r-variables, $w_1,\dots, w_q$ are extra v-variables, and where the differential polynomials 
 $P_1, \ldots, P_p, Q_1, \ldots, Q_q \in \qq\{ x \}$ have all their coefficients in $\zz$. Note that a special formula contains no quantified $f$-variables. 

\medskip\noindent
We now state the relative quantifier elimination result.

\begin{thm} \label{QuantifierElimination}
Every $L_3(\textnormal{ac})$-formula $\theta(x, y, z)$ is $T$-equivalent to some special $L_3(\textnormal{ac})$-formula
$\theta^*(x,y,z)$.
\end{thm}

\begin{proof}
For a model $\mathcal{K} = (K, \mathbf{k}, \Gamma; \pi, v, c, \textnormal{ac})$ of $T$ and for a tuple $(a, d, \gamma)$ where $a\in K^l$, $d\in \mathbf{k}^m$ and $\gamma\in \Gamma^n$,
define the {\em special type of $(a, d, \gamma)$} (in $\mathcal{K}$), denoted by $\textnormal{sptp}(a, d, \gamma)$, to be the following set of special formulas:
$$\textnormal{sptp}(a, d, \gamma)\ =\ \{ \phi(x, y, z)\ :\ \phi \textnormal{ is a special } L_3(\textnormal{ac}) \textnormal{-formula and }\mathcal{K} \models \phi(a, d, \gamma) \}.$$
Let $\mathcal{K}_1$ and $\mathcal{K}_2$ be models of $T$ and let $a_i\in K_i^l$, $d_i \in \mathbf{k}_i^m$ and $\gamma_i \in \Gamma_i^n$ for $i = 1, 2$ be such that $(a_1, d_1, \gamma_1)$ and $(a_2, d_2, \gamma_2)$ have the same special type, in $\mathcal{K}_1$ and $\mathcal{K}_2$ respectively.
By Stone's Representation Theorem, it is enough to show that then $(a_1, d_1, \gamma_1)$ and $(a_2, d_2, \gamma_2)$ realize the same type in $\mathcal{K}_1$ and $\mathcal{K}_2$, respectively. 
%We now construct good subfields $\mathbf{E}_1$ and $\mathbf{E}_2$ of $\mathcal{K}_1$ and $\mathcal{K}_2$ respectively and a good map between them.
Consider the differential subfield $E_i := \qq \< a_i \>$ of $K_i$, the ordered subgroup $\Gamma_{\mathbf{E}_i}$ of $\Gamma_i$ generated by $v_i(E_i^\times)$ and $\gamma_i$, and the differential subfield $$\mathbf{k}_{\mathbf{E}_i}\  :=\ \qq \< \textnormal{ac}_i(E_i), c_i(\Gamma_{\mathbf{E}_i}), d_i\>$$ of $\mathbf{k}_i$, for $i = 1, 2$. Then 
$\mathbf{E}_i := (E_i, \mathbf{k}_{\mathbf{E}_i}, \Gamma_{\mathbf{E}_i})$
is a good substructure of $\mathcal{K}_i$, for $i=1,2$.  
Note that for all $P\in \qq\{x\}$ and for $i=1,2$ we have $P(a_i) = 0$ iff 
$\textnormal{ac}_i(P(a_i)) = 0$. Since $a_1$ and $a_2$ have the same special type, this yields a differential field isomorphism $f : E_1 \to E_2$ with $f(a_1) = a_2$. 
It is also routine to show that we have
an ordered group isomorphism $f_v : \Gamma_{\mathbf{E}_1} \to \Gamma_{\mathbf{E}_2}$ such that $f_v(\gamma_1) = \gamma_2$ and $f_v(v_1(a)) = v_2(f(a))$ for all $a \in E_1^\times$, and a differential field isomorphism
$f_r : \mathbf{k}_{\mathbf{E}_1} \to \mathbf{k}_{\mathbf{E}_2}$ with $f_r(d_1) = d_2$, $f_r(c_1(\gamma)) = c_2(f_v(\gamma))$ for all $\gamma \in \Gamma_{\mathbf{E}_1}$ and $f_r(\textnormal{ac}_1(a)) = \textnormal{ac}_2(f(a))$ for all $a \in E_1$. These properties of the maps
$f, f_r, f_v$ only use that the tuples
$(a_1, d_1, \gamma_1)$ and $(a_2, d_2, \gamma_2)$ realize the same {\em quantifier-free\/} special formulas $\phi(x,y,z)$ in $\mathcal{K}_1$ and $\mathcal{K}_2$ respectively, but the full assumption on these two tuples guarantees that $(f_r, f_v)$ is a partial elementary map between $(\mathbf{k}_1, \Gamma_1; c_1)$ and $(\mathbf{k}_2, \Gamma_2; c_2)$.
Thus we have a good map $\mathbf{f} = (f, f_r, f_v)$, and
it remains to apply Theorem \ref{eqthm}.
\end{proof}

\begin{cor}
Let $\mathcal{K} = (K, \mathbf{k}, \Gamma; \pi, v, c, \textnormal{ac}) \models T$. Then any set $X\subseteq \mathbf{k}^m\times \Gamma^n$ that is definable in $\mathcal{K}$ is definable in the $L_{\textnormal{rv}}$-structure $(\mathbf{k}, \Gamma; c)$. In particular, $(\mathbf{k}, \Gamma; c)$ is stably embedded in $\mathcal{K}$. 
\end{cor}

\medskip\noindent
The angular component map does not occur in the above reduct
$(\mathbf{k}, \Gamma; c)$, so we can eliminate it in the result above in view of Corollary~\ref{buildac}:

\begin{cor}
Let $\mathcal{K} = (K, \mathbf{k}, \Gamma; \pi, v, c)$ be a
$\d$-henselian model of $\Mo(c)$. Then any set $X\subseteq \mathbf{k}^m\times \Gamma^n$ that is definable in $\mathcal{K}$ is definable in the $L_{\textnormal{rv}}$-structure $(\mathbf{k}, \Gamma; c)$. 
%In particular, $(\mathbf{k}, \Gamma; c)$ is stably embedded in $\mathcal{K}$. 
\end{cor}

%%%%%%%%%%%%%%%%%%%%%%%%%%%%%%%%%%%%%%%%%%%%%%%%%%%%%%%%%%%%%%%%%%%%%%%%%%
\section{NIP}

\newcommand{\tikzoverset}[2]{%
  \tikz[baseline=(X.base),inner sep=0pt,outer sep=0pt]{%
    \node[inner sep=0pt,outer sep=0pt] (X) {$#2$}; 
    \node[yshift=1pt] at (X.north) {$#1$};
}}

\newcommand{\harpoon}{\tikzoverset{\rightharpoonup}}

\noindent
Let  $\mathcal{K} = (K, \mathbf{k}, \Gamma; \pi, v, c, \textnormal{ac})$ be a model of $T$. When does $\mathcal{K}$ have NIP (the Non-Independence Property)? This can be reduced to the same question for $(\mathbf{k}, \Gamma; c)$:
%Let $x$ and $\tilde{x}$ be finite tuples of f-variables, $y$ and $\tilde{y}$ be finite tuples of r-variables, $z$ and $\tilde{z}$ be finite tuples of v-variables.
%By $|w|$ we will denote the length of $w$ in its ambient space. 

\begin{cor}
$\mathcal{K}$ has $\textnormal{NIP}$ if and only if the $L_{\textnormal{rv}}$-structure $(\mathbf{k}, \Gamma; c)$ has $\textnormal{NIP}$. 
\end{cor}

\begin{proof}
The forward direction is clear. To prove the contrapositive of the other direction, assume $\varphi(x, y, z; \tilde{x}, \tilde{y}, \tilde{z})$ is an $L_3(\textnormal{ac})$-formula having IP in $\mathcal{K}$, with $|x| = k$, $|y| = l$, $|z| = m$, $|\tilde{x}| = \tilde{k}$, $|\tilde{y}| = \tilde{l}$, $|\tilde{z}| = \tilde{m}$.
Moreover, without loss of generality, we can assume $\mathcal{K}$ is $2^{\aleph_0}$-saturated.
This means we have a sequence $\{(a_i, u_i, \gamma_i)\}_{i\in \nn}$ with $a_i\in K^k$, $u_i\in \mathbf{k}^l$ and $\gamma_i\in \Gamma^m$ and for every $I \subseteq \nn$, tuples $\tilde{a}_I\in K^{\tilde{k}}$, $\tilde{u}_I\in \mathbf{k}^{\tilde{l}}$ and $\tilde{\gamma}_I\in \Gamma^{\tilde{m}}$, such that for all $i\in\nn$ and $I\subseteq \nn$, 
$$\mathcal{K} \models \varphi(a_i, u_i, \gamma_i; \tilde{a}_I, \tilde{u}_I, \tilde{\gamma}_I) \Leftrightarrow i \in I.$$
By Theorem~\ref{QuantifierElimination}, $\varphi(x, y, z; \tilde{x}, \tilde{y}, \tilde{z})$ is $T$-equivalent to a special formula
$$ \psi\Big(\textnormal{ac}\big(\harpoon{P}(x)\big), v\big(\harpoon{Q}(x)\big), y, z; \textnormal{ac}\big(\harpoon{R}(\tilde{x})\big), v\big(\harpoon{S}(\tilde{x})\big), \tilde{y}, \tilde{z}\Big),$$
where $\psi$ is an $L_{\textnormal{rv}}$-formula and $\harpoon{P}$, $\harpoon{Q}$, $\harpoon{R}$, $\harpoon{S}$ are  finite tuples of differential polynomials in $\qq\{\tilde{x}\}$.
This yields tuples witnessing that $\psi$ has IP in 
$(\mathbf{k}, \Gamma; c)$:  
\begin{comment} take our sequence to be $\Big(\text{ac}\big(\harpoon{P}(a_i)\big), v\big( \harpoon{Q}(a_i) \big), u_i, \gamma_i \Big)$ and for every $I\subseteq \nn$, take the elements $\text{ac}\big(\harpoon{P}(\tilde{a}_I)\big)$, $v\big( \harpoon{Q}(\tilde{a}_I) \big)$, $\tilde{u}_I$ and $\tilde{\gamma}_I$.
Note that
\end{comment}
$$(\mathbf{k}, \Gamma; c) \models \psi\Big(\textnormal{ac}\big(\harpoon{P}(a_i)\big), v\big(\harpoon{Q}(a_i)\big), u_i, \gamma_i; \textnormal{ac}\big(\harpoon{R}(\tilde{a}_I)\big), v\big(\harpoon{S}(\tilde{a}_I)\big), \tilde{u}_I, \tilde{\gamma}_I\Big)
\Leftrightarrow i \in I,$$
for all $i\in \nn$ and $I\subseteq \nn$.
\end{proof}

\noindent
By Corollary~\ref{buildac} the result just proved goes through for $\d$-henselian model of $\Mo(c)$. 

\medskip\noindent
{\em Example of a $\d$-henselian model of 
$\Mo(c)$ with few constants that has NIP}. Let $K := \T[\mathrm{i}]((t^\R))_c$ be the $\d$-henselian monotone valued differential field considered in \cite[Section~4]{TH}; here $\T$ is the valued differential field of transseries, $\mathrm{i}^2 = -1$, and $c:\mathbb{R} \to \T[\mathrm{i}]$ is the additive map given by $c(r) = \mathrm{i} r$.
% with the property that $c(\R) \cap \T[\mathrm{i}]^\dagger = \{ 0 \}$.
By \cite[Proposition~16.6.6]{ADAMTT}, $\T$ has NIP. Then the $L_{\textnormal{rv}}$-structure $(\T[\mathrm{i}], \R; c)$ has NIP, since it is interpretable in the valued differential field $\T$. Therefore, $\mathcal{K} = (K, \T[\mathrm{i}], \R; \pi, v, c)$, where $\pi$ and $v$ are the obvious maps, also has NIP.  
%and $\textnormal{ac}$ is the angular component constructed in Section 3, also has NIP.

%We finish with an example of a model $\mathcal{K}$ of $T$ as above where $\mathbf{k}$ and $\Gamma$ have NIP, but $(\mathbf{k}, \Gamma; c)$ does not. We begin with the 2-sorted $L_{\textnormal{rv}}$-structure $(\mathbf{k}, \Gamma; c):= (\T, \qq; c)$ where $c:\qq\to \T$ is the inclusion map. Then $\T$ as a differential field and $\qq$ as an ordered abelian group have NIP, but $(\T, \qq; c)$ does not, since it defines $\zz$. Then $\mathcal{K}:= (\T((t^\qq))_c, \T, \qq; \pi, v, c, \textnormal{ac})$, where $\pi$ and $v$ are the obvious maps and $\textnormal{ac}$ is the angular component constructed in Section 3, gives us the desired example.

%Note that for the theory of $(\mathbf{k}, \Gamma; c)$ to have NIP, it might not be enough for the theories of $\mathbf{k}$ and $\Gamma$ to have NIP. For example, consider $(\T, \qq; c)$ where $c$ is the inclusion map. Then the theory of $\mathbb{\T}$ as a differential field and the theory of $\qq$ as an ordered abelian group have NIP, however, $(\T, \qq; c)$ defines $\zz$ and hence, its theory has IP.

\section*{Acknowledgments}
\noindent
The author would like to thank Lou van den Dries for numerous discussions and comments on this paper.

\end{document}